\documentclass[a4paper,11pt]{article}%
\usepackage{graphicx}
\usepackage{amsmath,amsxtra,amssymb,latexsym, amscd,amsthm}
\usepackage[mathscr]{eucal}
\newtheorem{thm}{Theorem}[section]
\newtheorem{cor}{Corollary}[section]
\newtheorem{lem}{Lemma}[section]
\newtheorem{prop}{Proposition}[section]
\theoremstyle{definition}

\newtheorem{exam}{Example}[section]

\theoremstyle{remark}
\newtheorem{rem}{Remark}[section]
\numberwithin{equation}{section}

\begin{document}
\title{Some new concentration inequalities for the It\^o stochastic integral}
\author{Nguyen Tien Dung\thanks{Department of Mathematics, University of Science, Vietnam National University, Hanoi, 334 Nguyen
Trai, Thanh Xuan, Hanoi, 084 Vietnam. Email: dung@hus.edu.vn}}

\date{March 6, 2024}

\maketitle
\begin{abstract} In this paper, based on the techniques of Malliavin calculus, we provide some new concentration inequalities for the running supremum of the It\^o stochastic integral with unbounded integrands. Several applications and examples are provided as well.
\end{abstract}

\section{Introduction}Let $(B_t)_{t\in [0,T]}$ be a standard Brownian motion defined on a complete probability space $(\Omega,\mathcal{F},\mathbb{F},P),$ where $\mathbb{F}=(\mathcal{F}_t)_{t\in [0,T]}$ is a natural filtration generated $B.$ We consider the It\^o stochastic integrals
$$\int_0^tu_sdB_s,\,\,0\leq t\leq T,$$
where $u=(u_t)_{t\in [0,T]}$ is an $\mathbb{F}$-adapted stochastic process such that $\int_0^TE|u_s|^2ds<\infty.$ The It\^o stochastic integrals are the most fundamental research objects in Stochastic analysis and they, of course, have many fruitful properties. In the present paper, we revisit the concentration property of It\^o stochastic integrals. Let us recall the following classical result: If $u$ is bounded in $L^2[0,T],i.e.$
\begin{equation}\label{igy}
\int_0^Tu_s^2ds\leq M^2\,\,\,a.s.
\end{equation}
for some deterministic constant $M,$ then we have
\begin{equation}\label{vms}
P\left(\sup\limits_{0\leq t\leq T}\int_0^tu_sdB_s>x\right)\leq e^{-\frac{x^2}{2M^2}},\,\,x\geq 0.
\end{equation}
So a natural question arising here is that how to handle the case of unbounded integrands? Surprisingly, to the best of our knowledge, this question is still open. Motivated by this observation, our aim is partially fill up this gap. The key idea is to replace the deterministic constant $M$ by a random variable. More specifically, we replace the condition (\ref{igy}) by a weaker one that reads
\begin{equation}\label{vms1}
\int_0^Tu_t^2dt\leq \int_0^T\bar{u}_t^2dt\,a.s.
\end{equation}
where $\bar{u}=(\bar{u}_t)_{t\in [0,T]}$ is an $\mathbb{F}$-adapted stochastic process satisfying, for each $t\in [0,T],$ $\bar{u}_t$ is a Malliavin differentiable random variable. This assumption allows us to employ the techniques of Malliavin calculus and we obtain new concentration inequalities for a general class of unbounded integrands. Our abstract results will be stated in Theorems \ref{yyj} and \ref{yyj2s}. A special consequence of these results is the following, see Corollary \ref{fdl} below: If the Malliavin derivative of $\bar{u}$ is bounded in $L^2([0,T]^2),i.e.$
\begin{equation}\label{fvsbs}
\int_0^T\int_0^s |D_r\bar{u}_s|^2drds\leq c\,\,a.s.
\end{equation}
for some $c>0,$ then we have
\begin{equation}\label{fvdt}
P\left(\sup\limits_{0\leq t\leq T}\int_0^tu_sdB_s>x\right)\leq 2\exp\left(-\frac{2x^2}{\left(\bar{\sigma}_T+\sqrt{\bar{\sigma}_T^2+4\sqrt{c}\,x}\right)^2}\right),\,\,\,x\geq  0,
\end{equation}
where $\bar{\sigma}_T^2:=\int_0^TE|\bar{u}_t|^2dt$ and $D$ denotes the Malliavin derivative operator. It is worth noting that (\ref{fvdt}) is a natural generalization of (\ref{vms}). Indeed, assuming the condition (\ref{igy}), the condition (\ref{vms1}) is satisfied with $\bar{u}_t=\frac{M}{\sqrt{T}},0\leq t\leq T.$ Furthermore, the condition (\ref{fvsbs}) is satisfied for all $c>0.$ Letting $c\to 0,$ the inequality (\ref{fvdt}) becomes
$$P\left(\sup\limits_{0\leq t\leq T}\int_0^tu_sdB_s>x\right)\leq 2e^{-\frac{x^2}{2\bar{\sigma}_T^2}}=2e^{-\frac{x^2}{2M^2}},\,\,x\geq 0,$$
which recovers (\ref{vms}) (up to a constant).

The rest of this article is organized as follows. In Section \ref{oo4}, we recall some concepts of Malliavin calculus and a general estimate for the tail distribution of Malliavin differentiable random variables. Our main results are then stated and proved in Section \ref{fpq}. In Section \ref{gvs}, we apply our results to derive new concentration inequalities for Gaussian functionals, double Wiener-It\^o integrals and CIR processes.

\section{Preliminaries}\label{oo4}
Let us recall some elements of stochastic calculus of variations (for more details see \cite{nualartm2}). For $h\in L^2[0,T],$ we denote by $B(h)$ the Wiener integral
$$B(h)=\int\limits_0^T h(t)dB_t.$$
Let $\mathcal{S}$ denote the dense subset of $L^2(\Omega, \mathcal{F},P)$ consisting of smooth random variables of the form
\begin{equation}\label{ro}
F=f(B(h_1),...,B(h_n)),
\end{equation}
where $n\in \mathbb{N}, f\in C_b^\infty(\mathbb{R}^n),h_1,...,h_n\in L^2[0,T].$ If $F$ has the form (\ref{ro}), we define its Malliavin derivative as the process $DF:=\{D_tF, t\in [0,T]\}$ given by
$$D_tF=\sum\limits_{k=1}^n \frac{\partial f}{\partial x_k}(B(h_1),...,B(h_n)) h_k(t).$$
For any $1\leq p<\infty,$ we shall denote by $\mathbb{D}^{1,p}$ 
the closure of $\mathcal{S}$ with respect to the norm
$$\|F\|^p_{1,p}:=E|F|^p+E\bigg[\int\limits_0^T|D_u F|^pdu\bigg].$$
A random variable $F$ is said Malliavin differentiable if it belongs to $\mathbb{D}^{1,2}.$ The following concentration inequality will play a key role in the present work.
\begin{lem}\label{kf98jk0} Let $Z$ be a centered random variable in $\mathbb{D}^{1,2}.$ Assume there exists a non-random constant $M$ such that
 \begin{equation}\label{lowup02}
\int_0^T |D_rZ|^2dr\leq M^2\,\,a.s.
\end{equation}
Then the following estimate for tail probabilities holds:
\begin{equation}\label{lowup01}
P\left(Z\geq x\right)\leq e^{-\frac{x^2}{2M^2}},\quad x\geq 0.
\end{equation}
\end{lem}
\begin{proof}
See Theorem 9.1.1 in \cite{Ustunel2015}.
\end{proof}

\section{Main results}\label{fpq}
Our first lemma is probably well known. It is stated and proved here for the reader's convenience.
\begin{lem}
Let $u=(u_t)_{t\in [0,T]}$ be an $\mathbb{F}$-adapted stochastic process with $\int_0^TE|u_s|^2ds<\infty.$ Then, for any $y>0,$ we have
\begin{equation}\label{pwcl}
P\left(\sup\limits_{0\leq t\leq T}\int_0^tu_sdB_s>x\right)\leq e^{-\frac{x^2}{2y}}+P\left(\int_0^T u_s^2ds\geq y\right),\,\,\,x\geq 0.
\end{equation}
\end{lem}
\begin{proof}
Let $\lambda$ be a positive real number. We have, for all $x\geq 0$ and $y>0,$
\begin{align*}
P&\left(\sup\limits_{0\leq t\leq T}\int_0^tu_sdB_s>x\right)\\
&\leq P\left(\sup\limits_{0\leq t\leq T}\int_0^t\lambda u_sdB_s>\lambda x,\int_0^T u_s^2ds<y\right)+P\left(\int_0^T u_s^2ds\geq y\right)\\
&\leq P\left(\sup\limits_{0\leq t\leq T}\left(\int_0^t\lambda u_sdB_s-\frac{1}{2}\int_0^t \lambda^2 u_s^2ds\right)>\lambda x-\frac{\lambda^2 y}{2}\right)+P\left(\int_0^T u_s^2ds\geq y\right).
\end{align*}
By the maximal inequality for supermartingale we have
\begin{align*}
P&\left(\sup\limits_{0\leq t\leq T}\left(\int_0^t\lambda u_sdB_s-\frac{1}{2}\int_0^t \lambda^2 u_s^2ds\right)>\lambda x-\frac{\lambda^2 y}{2}\right)\\
&\leq P\left(\sup\limits_{0\leq t\leq T}\exp\left(\int_0^t\lambda u_sdB_s-\frac{1}{2}\int_0^t \lambda^2 u_s^2ds\right)>\exp\left(\lambda x-\frac{\lambda^2 y}{2}\right)\right)\\
&\leq \exp\left(-\lambda x+\frac{\lambda^2 y}{2}\right),\,\,\,x\geq 0.
\end{align*}
The function $\lambda\mapsto \exp\left(-\lambda x+\frac{\lambda^2 y}{2}\right)$ attains its minimum value at $\lambda_0=\frac{x}{y}.$ Choosing $\lambda=\lambda_0,$ we get
$$P\left(\sup\limits_{0\leq t\leq T}\left(\int_0^t\lambda_0 u_sdB_s-\frac{1}{2}\int_0^t \lambda_0^2 u_s^2ds\right)>\lambda_0 x-\frac{\lambda_0^2 y}{2}\right)\leq e^{-\frac{x^2}{2y}},\,\,\,x\geq 0,$$
and hence,
$$P\left(\sup\limits_{0\leq t\leq T}\int_0^tu_sdB_s>x\right)\leq e^{-\frac{x^2}{2y}}+P\left(\int_0^T u_s^2ds\geq y\right),\,\,\,x\geq 0.$$
The proof of the lemma is complete.
\end{proof}
We now use the techniques of Malliavin calculus to derive new concentration inequalities for  It\^o stochastic integrals.
\begin{thm}\label{yyj} Let $u=(u_t)_{t\in [0,T]}$ be an $\mathbb{F}$-adapted stochastic process verifying the following hypothesis: there exists an $\mathbb{F}$-adapted stochastic process $\bar{u}=(\bar{u}_t)_{t\in [0,T]}$ such that $\bar{\sigma}_T^2:=\int_0^TE|\bar{u}_s|^2ds<\infty,$ $\int_0^Tu_t^2dt\leq \int_0^T\bar{u}_t^2dt\,a.s.$ and $\bar{u}_t\in \mathbb{D}^{1,2}$ for each $t\in [0,T]$ and
\begin{equation}\label{cml}
\int_0^T\int_0^s |D_r\bar{u}_s|^2drds\leq c\left(\int_0^T \bar{u}_s^2ds\right)^{\alpha}\,\,a.s.
\end{equation}
for some $c>0$ and $\alpha\in [0,1].$ For any increasing function $v:[0,\infty)\to [1,\infty),$ we have

\noindent (i) If $\alpha\in [0,1),$ then
\begin{equation}\label{ttr}
P\left(\sup\limits_{0\leq t\leq T}\int_0^tu_sdB_s>x\right)\leq e^{-\frac{x^2}{2\bar{\sigma}_T^2v(x)}}+\exp\left(-\frac{\bar{\sigma}_T^{2-2\alpha}(v^\frac{1-\alpha}{2}(x)-1)^2}{2c(1-\alpha)^2}\right),\,x\geq  v^{-1}(1).
\end{equation}
\noindent (ii) If $\alpha=1,$ then
\begin{equation}\label{8ujk}
P\left(\sup\limits_{0\leq t\leq T}\int_0^tu_sdB_s>x\right)\leq e^{-\frac{x^2}{2\bar{\sigma}_T^2v(x)}}+\exp\left(-\frac{\ln^2(v(x))}{8c}\right),\,\,\,x\geq  v^{-1}(1),
\end{equation}
where $v^{-1}$ denotes the inverse function of $v.$
\end{thm}
\begin{proof} Applying the inequality (\ref{pwcl}) to $y=\bar{\sigma}_T^2v(x)$ we obtain
\begin{align}
P\left(\sup\limits_{0\leq t\leq T}\int_0^tu_sdB_s>x\right)&\leq e^{-\frac{x^2}{2\bar{\sigma}_T^2v(x)}}+P\left(\int_0^T u_s^2ds\geq \bar{\sigma}_T^2v(x)\right)\notag\\
&\leq e^{-\frac{x^2}{2\bar{\sigma}_T^2v(x)}}+P\left(\int_0^T \bar{u}_s^2ds\geq \bar{\sigma}_T^2v(x)\right),\,\,\,x\geq 0.\label{ppl}
\end{align}
In order to bound the last addend in the right hand side of (\ref{ppl}), we consider two cases.

\noindent{\it Case 1: $0\leq \alpha<1.$} Consider the random variable
$$X_\alpha:=\left(\int_0^T \bar{u}_s^2ds\right)^{\frac{1-\alpha}{2}}.$$
By the chain rule of Malliavin derivatives we have
$$D_rX_\alpha=(1-\alpha)\int_r^T \bar{u}_sD_r\bar{u}_sds\left(\int_0^T \bar{u}_s^2ds\right)^{\frac{-1-\alpha}{2}},\,\,0\leq r\leq T.$$
Then, by the Cauchy-Schwarz inequality, we obtain
$$|D_rX_\alpha|^2\leq (1-\alpha)^2\int_r^T |D_r\bar{u}_s|^2ds\left(\int_0^T \bar{u}_s^2ds\right)^{-\alpha}.$$
This, combined with the condition (\ref{cml}), yields
\begin{align}
\int_0^T|D_rX_\alpha|^2dr&\leq (1-\alpha)^2\int_0^T\int_0^s |D_r\bar{u}_s|^2drds\left(\int_0^T \bar{u}_s^2ds\right)^{-\alpha}\notag\\
&\leq c(1-\alpha)^2\,\,a.s.\label{udb}
\end{align}
We observe that, by Jensen's inequality,  $E[X_\alpha]\leq \left(\int_0^T E[\bar{u}_s^2]ds\right)^{\frac{1-\alpha}{2}}=\bar{\sigma}_T^{1-\alpha}.$ We therefore get
\begin{align*}
P\left(\int_0^T \bar{u}_s^2ds\geq \bar{\sigma}_T^2v(x)\right)&=P\left(X_\alpha\geq \bar{\sigma}_T^{1-\alpha}v^\frac{1-\alpha}{2}(x)\right)\\
&\leq P\left(X_\alpha-E[X_\alpha]\geq \bar{\sigma}_T^{1-\alpha}v^\frac{1-\alpha}{2}(x)-\bar{\sigma}_T^{1-\alpha}\right),\,\,\,x\geq 0.
\end{align*}
In view of Lemma \ref{kf98jk0}, the relation (\ref{udb}) implies that
\begin{align*}
P\left(\int_0^T \bar{u}_s^2ds\geq \bar{\sigma}_T^2v(x)\right)&\leq \exp\left(-\frac{(\bar{\sigma}_T^{1-\alpha}v^\frac{1-\alpha}{2}(x)-\bar{\sigma}_T^{1-\alpha})^2}{2c(1-\alpha)^2}\right),\,\,\,x\geq  v^{-1}(1).
\end{align*}
Inserting the above estimate into (\ref{ppl}) gives us
\begin{equation*}\label{qkda}
P\left(\sup\limits_{0\leq t\leq T}\int_0^tu_sdB_s>x\right)\leq e^{-\frac{x^2}{2\bar{\sigma}_T^2v(x)}}+\exp\left(-\frac{(\bar{\sigma}_T^{1-\alpha}v^\frac{1-\alpha}{2}(x)-\bar{\sigma}_T^{1-\alpha})^2}{2c(1-\alpha)^2}\right),\,x\geq  v^{-1}(1).
\end{equation*}
This finishes the proof of (\ref{ttr}).

\noindent{\it Case 2: $\alpha=1.$} Consider the random variable
$$X_1:=\ln\left(\int_0^T \bar{u}_s^2ds\right).$$
Once again, by the chain rule of Malliavin derivatives, we have
$$D_rX_1=\frac{\int_r^T 2\bar{u}_sD_r\bar{u}_sds}{\int_0^T \bar{u}_s^2ds},\,\,0\leq r\leq T.$$
Then, by the Cauchy-Schwarz inequality, we obtain
$$|D_rX_1|^2\leq \frac{4\int_r^T |D_r\bar{u}_s|^2ds}{\int_0^T \bar{u}_s^2ds}\,\,a.s.$$
Recalling the condition (\ref{cml}) with $\alpha=1,$ we deduce
$$\int_0^T|D_rX_1|^2dr\leq \frac{4\int_0^T\int_0^s |D_r\bar{u}_s|^2drds}{\int_0^T \bar{u}_s^2ds}\leq 4c\,\,a.s.$$
Thus the random variable $X_1$ satisfies the conditions  of Lemma \ref{kf98jk0}. Consequently, note that $E[X_1]\leq \ln(\bar{\sigma}_T^2),$ we obtain
\begin{align*}
P\left(\int_0^T \bar{u}_s^2ds\geq \bar{\sigma}_T^2v(x)\right)&=P\left(X_1\geq \ln(\bar{\sigma}_T^2v(x))\right)\\
&\leq P\left(X_1-E[X_1]\geq \ln(\bar{\sigma}_T^2v(x))-\ln(\bar{\sigma}_T^2)\right)\\
&\leq \exp\left(-\frac{\ln^2(v(x))}{8c}\right),\,\,\,x\geq  v^{-1}(1)
\end{align*}
Inserting the above estimate into (\ref{ppl}) gives us
\begin{equation*}\label{qkd}
P\left(\sup\limits_{0\leq t\leq T}\int_0^tu_sdB_s>x\right)\leq e^{-\frac{x^2}{2\bar{\sigma}_T^2v(x)}}+\exp\left(-\frac{\ln^2(v(x))}{8c}\right),\,\,\,x\geq  v^{-1}(1).
\end{equation*}
This finishes the proof of (\ref{8ujk}). So the proof of the theorem is complete.
\end{proof}
\begin{rem} When $\alpha\in [0,1),$ in view of (\ref{ttr}), we can choose $v(x)$ solving the following equation
$$\frac{x^2}{2\bar{\sigma}_T^2v(x)}=\frac{(\bar{\sigma}_T^{1-\alpha}v^\frac{1-\alpha}{2}(x)-\bar{\sigma}_T^{1-\alpha})^2}{2c(1-\alpha)^2},\,\,x\geq 0,$$
or equivalently
\begin{equation}\label{akola}
\frac{x}{\bar{\sigma}_T\sqrt{v(x)}}=\frac{\bar{\sigma}_T^{1-\alpha}v^\frac{1-\alpha}{2}(x)-\bar{\sigma}_T^{1-\alpha}}{\sqrt{c}(1-\alpha)},\,\,x\geq 0.
\end{equation}
Similarly, when $\alpha=1,$ we can choose the function $v(x)$ such that
\begin{equation}\label{kola}
\sqrt{v(x)}\ln v(x)=\frac{2\sqrt{c}}{\bar{\sigma}_T}x,\,\,\,x\geq 0.
\end{equation}
Note that, for each $x\geq 0,$ the equations (\ref{akola}) and (\ref{kola}) admit a unique solution $v(x).$ Moreover, $v(0)=1$ and $v(x)$ is an increasing function.
\end{rem}
In special case, when $\alpha=0,$  the equation (\ref{akola}) can be solved explicitly and its unique solution is given by
$$v(x)=\frac{1}{4\bar{\sigma}_T^2}\left(\bar{\sigma}_T+\sqrt{\bar{\sigma}_T^2+4\sqrt{c}\,x}\right)^2,\,\,\,x\geq 0.$$
Using this unique solution and the concentration inequality (\ref{ttr}), we obtain the following.
\begin{cor}\label{fdl}
Let $u=(u_t)_{t\in [0,T]}$ be an $\mathbb{F}$-adapted stochastic process verifying the following hypothesis: there exists an $\mathbb{F}$-adapted stochastic process $\bar{u}=(\bar{u}_t)_{t\in [0,T]}$ such that $\bar{\sigma}_T^2:=\int_0^TE|\bar{u}_s|^2ds<\infty,$ $\int_0^Tu_t^2dt\leq \int_0^T\bar{u}_t^2dt\,a.s.$ and $\bar{u}_t\in \mathbb{D}^{1,2}$ for each $t\in [0,T]$ and
\begin{equation}\label{crml}
\int_0^T\int_0^s |D_r\bar{u}_s|^2drds\leq c\,\,a.s.
\end{equation}
for some $c>0.$ Then, it holds that
\begin{equation}\label{ydlj}
P\left(\sup\limits_{0\leq t\leq T}\int_0^tu_sdB_s>x\right)\leq 2\exp\left(-\frac{2x^2}{\left(\bar{\sigma}_T+\sqrt{\bar{\sigma}_T^2+4\sqrt{c}\,x}\right)^2}\right),\,\,\,x\geq  0.
\end{equation}
\end{cor}
In general case, it is not easy to solve the equations (\ref{akola}) and (\ref{kola}). Here, by using a suitable function $v(x),$ we obtain the following explicit bounds.
\begin{cor}\label{ryyj} Let $u=(u_t)_{t\in [0,T]}$ be as in Theorem \ref{yyj}. The following statements hold true.

\noindent (i) If $\alpha\in [0,1),$ then
 \begin{multline}\label{vttr}
P\left(\sup\limits_{0\leq t\leq T}\int_0^tu_sdB_s>x\right)\leq 2\exp\left(-\frac{x^2}{2\left(\bar{\sigma}_T^{1-\alpha}+\left(c(1-\alpha)^2x^2\right)^{\frac{1-\alpha}{4-2\alpha}}\right)^{\frac{2}{1-\alpha}}}\right),\,\,\,x\geq 0.
\end{multline}
\noindent (ii) If $\alpha=1,$ then
\begin{multline}\label{eujk}
P\left(\sup\limits_{0\leq t\leq T}\int_0^tu_sdB_s>x\right)\leq \exp\left(-\frac{x^2\ln^2\left(\frac{\sqrt{c}}{\bar{\sigma}_T}x+e\right)}{2(\sqrt{c}\,x+\bar{\sigma}_T)^2}\right)
\\+\exp\left(-\frac{\left(\ln\left(\frac{\sqrt{c}}{\bar{\sigma}_T}x+1\right)-\ln\ln\left(\frac{\sqrt{c}}{\bar{\sigma}_T}x+e\right)\right)^2}{2c}\right),\,\,\,x\geq  0.
\end{multline}
\end{cor}
\begin{proof}We first recall that, when studying concentration inequalities, the asymptotic behavior as $x\to\infty$ is one of the most important characteristics.
Hence, when $\alpha\in [0,1),$ we should choose an increasing function $v$ satisfying $v(0)=1$ and $\frac{x^2}{2\bar{\sigma}_T^2v(x)}\sim\frac{(\bar{\sigma}_T^{1-\alpha}v^\frac{1-\alpha}{2}(x)-\bar{\sigma}_T^{1-\alpha})^2}{2c(1-\alpha)^2}$ as $x\to\infty.$ Here, among others, we use $v(x)=\left(1+\left(\frac{c(1-\alpha)^2x^2}{\bar{\sigma}_T^{4-2\alpha}}\right)^{\frac{1-\alpha}{4-2\alpha}}\right)^{\frac{2}{1-\alpha}}$ and we obtain (\ref{vttr}) directly from the concentration inequality (\ref{ttr}).

Similarly, when $\alpha=1,$ we want to  choose an increasing function $v$ satisfying $v(0)=1$ and $\frac{x^2}{2\bar{\sigma}_T^2v(x)}\sim\frac{\ln^2(v(x))}{8c}$ as $x\to\infty.$ Here we use $v(x)=\left(\frac{\sqrt{c}}{\bar{\sigma}_T}x+1\right)^2\ln^{-2}\left(\frac{\sqrt{c}}{\bar{\sigma}_T}x+e\right)
$ and hence, (\ref{8ujk}) becomes (\ref{eujk}).

The proof of the corollary is complete.
\end{proof}
\begin{rem} The estimates (\ref{vttr}) and (\ref{eujk}) give us the following asymptotic behaviors
$$\lim\limits_{x\to\infty}\frac{\ln P\left(\sup\limits_{0\leq t\leq T}\int_0^tu_sdB_s>x\right)}{x^{\frac{2-2\alpha}{2-\alpha}}}\leq -\frac{1}{2c^{\frac{1}{2-\alpha}}(1-\alpha)^{\frac{2}{2-\alpha}}}\,\,\,\text{when}\,\,\alpha\in [0,1),$$
$$\lim\limits_{x\to\infty}\frac{\ln P\left(\sup\limits_{0\leq t\leq T}\int_0^tu_sdB_s>x\right)}{\ln^2x}\leq -\frac{1}{2c}\,\,\,\text{when}\,\,\alpha=1.$$
\end{rem}

\begin{thm}\label{yyj2s} Let $u=(u_t)_{t\in [0,T]}$ be an $\mathbb{F}$-adapted stochastic process verifying the following hypothesis: there exists an $\mathbb{F}$-adapted stochastic process $\bar{u}=(\bar{u}_t)_{t\in [0,T]}$ such that $\bar{\sigma}_{\alpha,T}^2:=\int_0^TE|\bar{u}_s|^{2-2\alpha}ds<\infty,$ $\int_0^Tu_t^2dt\leq \int_0^T\bar{u}_t^2dt\,a.s.$ and $\bar{u}_t\in \mathbb{D}^{1,2}$ for each $t\in [0,T]$ and
\begin{equation}\label{frb}
 \int_0^T|\bar{u}_s|^{-2\alpha}\int_0^s |D_r\bar{u}_s|^2drds\leq c\,\,a.s.
\end{equation}
for some $c>0$ and $\alpha<0.$ For any increasing function $v:[0,\infty)\to [1,\infty),$ we have
\begin{multline}\label{erv}
P\left(\sup\limits_{0\leq t\leq T}\int_0^tu_sdB_s>x\right)\\\leq \exp\left(-\frac{x^2}{2\bar{\sigma}_{\alpha,T}^{\frac{2}{1-\alpha}}T^{-\frac{\alpha}{1-\alpha}}v(x)}\right)
+\exp\left(-\frac{\bar{\sigma}_{\alpha,T}^2(v^\frac{1-\alpha}{2}(x)-1)^2}{2c(1-\alpha)^2}\right),\,x\geq  v^{-1}(1).
\end{multline}
\end{thm}
\begin{proof}For each $\alpha<0,$ we put
$$X_\alpha:=\left(\int_0^T |\bar{u}_s|^{2-2\alpha}ds\right)^{\frac{1}{2}}.$$
We have
$$D_rX_\alpha=(1-\alpha)\int_r^T |\bar{u}_s|^{1-2\alpha}D_r|\bar{u}_s|ds\left(\int_0^T |\bar{u}_s|^{2-2\alpha}ds\right)^{-\frac{1}{2}},\,\,0\leq r\leq T.$$By using the Cauchy-Schwarz inequality and the fact $|D_r|\bar{u}_s||\leq |D_r\bar{u}_s|,$ we deduce
$$|D_rX_\alpha|^2\leq (1-\alpha)^2\int_r^T |\bar{u}_s|^{-2\alpha}|D_r\bar{u}_s|^2ds$$
which, combined with (\ref{frb}), yields
$$\int_0^T|D_rX_\alpha|^2dr\leq (1-\alpha)^2\int_0^T|\bar{u}|_s^{-2\alpha}\int_0^s |D_r\bar{u}_s|^2drds\leq c(1-\alpha)^2\,\,a.s.$$
So, for any increasing function $v:[0,\infty)\to [1,\infty),$ we can use Lemma \ref{kf98jk0} to get
$$P\left(X_\alpha-E[X_\alpha]\geq \bar{\sigma}_{\alpha,T}v^{\frac{1-\alpha}{2}}(x)-\bar{\sigma}_{\alpha,T}\right)\leq \exp\left(-\frac{(\bar{\sigma}_{\alpha,T}v^\frac{1-\alpha}{2}(x)-\bar{\sigma}_{\alpha,T})^2}{2c(1-\alpha)^2}\right),\,x\geq  v^{-1}(1).$$
On the other hand, by H\"older's inequality, we have
$$\int_0^T \bar{u}_s^2ds\leq T^{-\frac{\alpha}{1-\alpha}}\left(\int_0^T |\bar{u}_s|^{2-2\alpha}ds\right)^{\frac{1}{1-\alpha}}=T^{-\frac{\alpha}{1-\alpha}}X_\alpha^{\frac{2}{1-\alpha}},$$
and hence,
\begin{align*}
P\left(\int_0^T \bar{u}_s^2ds\geq \bar{\sigma}_{\alpha,T}^{\frac{2}{1-\alpha}}T^{-\frac{\alpha}{1-\alpha}}v(x)\right)&\leq P\left(X_\alpha\geq \bar{\sigma}_{\alpha,T}v^{\frac{1-\alpha}{2}}(x)\right)\\
&\leq P\left(X_\alpha-E[X_\alpha]\geq \bar{\sigma}_{\alpha,T}v^{\frac{1-\alpha}{2}}(x)-\bar{\sigma}_{\alpha,T}\right)\\
&\leq \exp\left(-\frac{(\bar{\sigma}_{\alpha,T}v^\frac{1-\alpha}{2}(x)-\bar{\sigma}_{\alpha,T})^2}{2c(1-\alpha)^2}\right),\,x\geq  v^{-1}(1).
\end{align*}
We now apply the inequality (\ref{pwcl}) to $y=\bar{\sigma}_{\alpha,T}^{\frac{2}{1-\alpha}}T^{-\frac{\alpha}{1-\alpha}}v(x)$ and we obtain
 \begin{multline*}
P\left(\sup\limits_{0\leq t\leq T}\int_0^tu_sdB_s>x\right)\leq \exp\left(-\frac{x^2}{2\bar{\sigma}_{\alpha,T}^{\frac{2}{1-\alpha}}T^{-\frac{\alpha}{1-\alpha}}v(x)}\right)+P\left(\int_0^T u_s^2ds\geq \bar{\sigma}_{\alpha,T}^{\frac{2}{1-\alpha}}T^{-\frac{\alpha}{1-\alpha}}v(x)\right)\\
\leq \exp\left(-\frac{x^2}{2\bar{\sigma}_{\alpha,T}^{\frac{2}{1-\alpha}}T^{-\frac{\alpha}{1-\alpha}}v(x)}\right)+P\left(\int_0^T \bar{u}_s^2ds\geq \bar{\sigma}_{\alpha,T}^{\frac{2}{1-\alpha}}T^{-\frac{\alpha}{1-\alpha}}v(x)\right)\\
\leq \exp\left(-\frac{x^2}{2\bar{\sigma}_{\alpha,T}^{\frac{2}{1-\alpha}}T^{-\frac{\alpha}{1-\alpha}}v(x)}\right)
+\exp\left(-\frac{(\bar{\sigma}_{\alpha,T}v^\frac{1-\alpha}{2}(x)-\bar{\sigma}_{\alpha,T})^2}{2c(1-\alpha)^2}\right),\,\,x\geq  v^{-1}(1).
\end{multline*}
This completes the proof of the theorem.
\end{proof}
\noindent By choosing an increasing function $v$ satisfying $v(0)=1$ and $\frac{x^2}{2\bar{\sigma}_{\alpha,T}^{\frac{2}{1-\alpha}}T^{-\frac{\alpha}{1-\alpha}}v(x)}\sim
\frac{(\bar{\sigma}_{\alpha,T}v^\frac{1-\alpha}{2}(x)-\bar{\sigma}_{\alpha,T})^2}{2c(1-\alpha)^2}$ as $x\to\infty,$  we obtain the following explicit bound.
\begin{cor}Let $u=(u_t)_{t\in [0,T]}$ be as in Theorem \ref{yyj2s}. It holds that
\begin{multline}\label{erv1}
P\left(\sup\limits_{0\leq t\leq T}\int_0^tu_sdB_s>x\right)\\\leq 2\exp\left(-\frac{x^2}{2T^{-\frac{\alpha}{2-\alpha}}
\left(\bar{\sigma}_{\alpha,T}T^{-\frac{\alpha}{4-2\alpha}}+\left(c(1-\alpha)^2x^2\right)^{\frac{1-\alpha}{4-2\alpha}}\right)^{\frac{2}{1-\alpha}}}\right),\,x\geq 0.
\end{multline}
Consequently, we have
$$\lim\limits_{x\to\infty}\frac{\ln P\left(\sup\limits_{0\leq t\leq T}\int_0^tu_sdB_s>x\right)}{x^{\frac{2-2\alpha}{2-\alpha}}}\leq -\frac{1}{2c^{\frac{1}{2-\alpha}}(1-\alpha)^{\frac{2}{2-\alpha}}T^{-\frac{\alpha}{2-\alpha}}}.$$
\end{cor}
\begin{proof}Follows directly from (\ref{erv}) with $v(x)=\left(1+\left(\frac{c(1-\alpha)^2x^2}{\bar{\sigma}_{\alpha,T}^{\frac{4-2\alpha}{1-\alpha}}T^{-\frac{\alpha}{1-\alpha}}}
\right)^{\frac{1-\alpha}{4-2\alpha}}\right)^{\frac{2}{1-\alpha}},x\geq 0.$
\end{proof}
In many applications, the random variables of interest are the form $\int_0^th(s,X_s)dB_s$ for some function $h$ and for some stochastic process $(X_t)_{t\in [0,T]}.$ We have the following.
\begin{cor}Let $X=(X_t)_{t\in [0,T]}$ be an $\mathbb{F}$-adapted stochastic process verifying the following hypotheses: $X_t\in \mathbb{D}^{1,2}$ for each $t\in [0,T]$ and
\begin{equation}\label{io2f}
|D_rX_s|\leq k(s,r),\,\,0\leq r\leq s\leq T,
\end{equation}
 for some integrable square deterministic function $k:[0,T]^2\to [0,\infty).$ Also let $f:[0,T]\times \mathbb{R}\to \mathbb{R}$ be such that $\frac{\partial f(t,x)}{\partial x}$ exists and
\begin{equation}\label{smv}
\bigg|\frac{\partial f(t,x)}{\partial x}\bigg|\leq L|f(t,x)|^\alpha\,\,\forall\,t\in[0,T],x\in \mathbb{R}
\end{equation}
for some $L>0$ and $\alpha\in (-\infty,1].$ Assume, in addition, that $\frac{\partial f(t,x)}{\partial x}$  has subexponential growth with respect to $x.$ Then, for any function $h$ satisfying $|h(t,x)|\leq |f(t,x)|\,\forall\,t\in[0,T]$ and $x\in \mathbb{R},$ we have

\noindent (i) If $\alpha<0,$ then
\begin{multline}\label{e6rv1}
P\left(\sup\limits_{0\leq t\leq T}\int_0^th(s,X_s)dB_s>x\right)\\\leq 2\exp\left(-\frac{x^2}{2T^{-\frac{\alpha}{2-\alpha}}
\left(\bar{\sigma}_{\alpha,T}T^{-\frac{\alpha}{4-2\alpha}}+\left(c(1-\alpha)^2x^2\right)^{\frac{1-\alpha}{4-2\alpha}}\right)^{\frac{2}{1-\alpha}}}\right),\,x\geq 0,
\end{multline}
where $c:=L^2\int_0^T\int_0^s k^2(s,r)drds$ and $\bar{\sigma}_{\alpha,T}^2:=\int_0^TE|f(s,X_s)|^{2-2\alpha}ds.$

\noindent (ii) If $\alpha\in [0,1),$ then
 \begin{multline}\label{v3ttr}
P\left(\sup\limits_{0\leq t\leq T}\int_0^th(s,X_s)dB_s>x\right)\\\leq 2\exp\left(-\frac{x^2}{2\left(\bar{\sigma}_T^{1-\alpha}+\left(c(1-\alpha)^2x^2\right)^{\frac{1-\alpha}{4-2\alpha}}\right)^{\frac{2}{1-\alpha}}}\right),\,\,\,x\geq 0,
\end{multline}
where $c:=L^2\left(\int_0^T\left(\int_0^s k^2(s,r)dr\right)^{\frac{1}{1-\alpha}}ds\right)^{1-\alpha}$ and $\bar{\sigma}_{T}^2:=\int_0^TE|f(s,X_s)|^2ds.$

\noindent (iii) If $\alpha=1,$ then
\begin{multline}\label{e3ujk}
P\left(\sup\limits_{0\leq t\leq T}\int_0^th(s,X_s)dB_s>x\right)\leq \exp\left(-\frac{x^2\ln^2\left(\frac{\sqrt{c}}{\bar{\sigma}_T}x+e\right)}{2(\sqrt{c}\,x+\bar{\sigma}_T)^2}\right)
\\+\exp\left(-\frac{\left(\ln\left(\frac{\sqrt{c}}{\bar{\sigma}_T}x+1\right)-\ln\ln\left(\frac{\sqrt{c}}{\bar{\sigma}_T}x+e\right)\right)^2}{2c}\right),\,\,\,x\geq  0,
\end{multline}
where $c:=L^2\sup\limits_{0\leq s\leq T}\int_0^s k^2(s,r)dr$ and $\bar{\sigma}_{T}^2:=\int_0^TE|f(s,X_s)|^2ds.$

\end{cor}
\begin{proof} We consider the stochastic processes
$$\text{$u_s=h(s,X_s)$ and $\bar{u}_s=f(s,X_s),0\leq s\leq T.$}$$
The condition (\ref{io2f}) and Lemma \ref{kf98jk0} imply
$$P\left(|X_s-EX_s|\geq x\right)\leq 2\exp\left(-\frac{x^2}{2\int_0^sk^2(s,r)dr}\right)\,\,\,\forall\,\,0< s\leq T,x\geq 0.$$
Hence, under the required conditions on $f,$ it is easy to see that  $\bar{u}_s\in \mathbb{D}^{1,2}$ for each $0\leq s\leq T.$ Moreover, we have
\begin{equation}\label{iof}
|D_r\bar{u}_s|=\big|k(s,r)f'(s,X_s)\big|\leq L|k(s,r)||\bar{u}_s|^\alpha,\,\,0\leq r\leq s\leq T.
\end{equation}
When $\alpha<0,$ the relation (\ref{iof}) implies that
$$\int_0^T|\bar{u}_s|^{-2\alpha}\int_0^s |D_r\bar{u}_s|^2drds\leq  L^2\int_0^T\int_0^s k^2(s,r)drds\,\,a.s.
$$
Hence, the condition (\ref{frb}) is satisfied with $c=L^2\int_0^T\int_0^s k^2(s,r)drds$ and the concentration inequality (\ref{e6rv1}) follows from (\ref{erv1}).

When $\alpha\in [0,1),$ by H\"older's inequality,  the relation (\ref{iof}) implies that
\begin{align*}
\int_0^T\int_0^s |D_r\bar{u}_s|^2drds&\leq L^2\int_0^T\int_0^s k^2(s,r)|\bar{u}_s|^{2\alpha}drds\\
&\leq L^2\left(\int_0^T\left(\int_0^s k^2(s,r)dr\right)^{\frac{1}{1-\alpha}}ds\right)^{1-\alpha}\left(\int_0^T \bar{u}_s^2ds\right)^{\alpha}.
\end{align*}
So the condition (\ref{cml}) is satisfied with $c=L^2\left(\int_0^T\left(\int_0^s k^2(s,r)dr\right)^{\frac{1}{1-\alpha}}ds\right)^{1-\alpha},$ and hence, the concentration inequality (\ref{v3ttr}) follows from (\ref{vttr}).

Similarly, when $\alpha=1,$ we have
\begin{align*}
\int_0^T\int_0^s |D_r\bar{u}_s|^2drds&\leq L^2\int_0^T\int_0^s k^2(s,r)|\bar{u}_s|^2drds\\
&\leq L^2\sup\limits_{0\leq s\leq T}\int_0^s k^2(s,r)dr\int_0^T|\bar{u}_s|^2ds.
\end{align*}
Thus the condition (\ref{cml}) is satisfied with $c= L^2\sup\limits_{0\leq s\leq T}\int_0^s k^2(s,r)dr$ and the concentration inequality (\ref{e3ujk}) follows from  (\ref{eujk}).

The proof of the proposition is complete.
\end{proof}
We end this section by providing a computational example.
\begin{exam}We consider the stochastic process $X_t=\max\limits_{0\leq u\leq t}B_u,0\leq t\leq 1.$ This process satisfies the condition (\ref{io2f}) because $|D_rX_s|\leq 1=:k(s,r),\,\,0\leq r\leq s\leq 1.$

\noindent(i) The function $f(t,x)=e^x,x\in \mathbb{R}$ satisfies the condition (\ref{smv}) with $\alpha=L=1.$ So, by using the inequality (\ref{e3ujk}) with $h(t,x)=e^x,$ we obtain
\begin{multline}
P\left(\sup\limits_{0\leq t\leq 1}\int_0^te^{\max\limits_{0\leq u\leq s}B_u}dB_s>x\right)\leq \exp\left(-\frac{x^2\ln^2\left(\frac{x}{\sqrt{2}\bar{\sigma}_T}+e\right)}{2(x+\bar{\sigma}_T)^2}\right)
\\+\exp\left(-\frac{1}{2}\left(\ln\left(\frac{x}{\bar{\sigma}_T}+1\right)-\ln\ln\left(\frac{x}{\bar{\sigma}_T}+e\right)\right)^2\right),\,\,\,x\geq  0,
\end{multline}
where $\bar{\sigma}_{T}^2:=\int_0^1E\big[e^{2\max\limits_{0\leq u\leq s}B_u}\big]ds.$

\noindent(ii) The function $f(t,x)=x^2,x\in \mathbb{R}$ satisfies the condition (\ref{smv}) with $\alpha=\frac{1}{2}$ and $L=2.$ So, by using the inequality (\ref{v3ttr}) with $h(t,x)=x^2,$ we obtain
$$
P\left(\sup\limits_{0\leq t\leq 1}\int_0^t(\max\limits_{0\leq u\leq s}B_u)^2dB_s>x\right)\leq 2\exp\left(-\frac{x^2}{2\left(2^{-\frac{1}{4}}+3^{-\frac{1}{12}}x^{\frac{1}{3}}\right)^{4}}\right),\,\,\,x\geq 0.
$$

\noindent(iii) For every $\varepsilon>0,$ the function $f(t,x)=\sqrt{x+\varepsilon},x\in \mathbb{R}_+$ satisfies the condition (\ref{smv}) with $\alpha=-1$ and $L=1.$ So, by using the inequality (\ref{e6rv1}) with $h(t,x)=\sqrt{x},$ we obtain
\begin{equation*}
P\left(\sup\limits_{0\leq t\leq 1}\int_0^t\sqrt{\max\limits_{0\leq u\leq s}B_u}dB_s>x\right)\leq 2\exp\left(-\frac{x^2}{2\bar{\sigma}_{\alpha,T}+2^{\frac{4}{3}}x^{\frac{2}{3}}}\right),\,x\geq 0,
\end{equation*}
where $\bar{\sigma}_{\alpha,T}^2:=\int_0^1E|\max\limits_{0\leq u\leq s}B_u+\varepsilon|^2ds.$ Letting $\varepsilon\to 0$ yields
$$P\left(\sup\limits_{0\leq t\leq 1}\int_0^t\sqrt{\max\limits_{0\leq u\leq s}B_u}dB_s>x\right)\leq 2\exp\left(-\frac{x^2}{\sqrt{2}+2^{\frac{4}{3}}x^{\frac{2}{3}}}\right),\,x\geq 0.$$
\end{exam}

\section{Applications}\label{gvs}
In this section, we provide some applications to illustrate the results obtained in the previous section.
\subsection{Gaussian functionals}
Let $Z=(Z_1,\cdots,Z_n)$ be an $n$-dimensional standard Gaussian vector. We consider Gaussian functionals of the form
$$F=f(Z_1,\cdots,Z_n),$$
where $f:\mathbb{R}^n\to \mathbb{R}$ is a deterministic function. It is well known that, when $f$ is a Lipschitz function with Lipschitz constant $L,$ the following Gaussian concentration inequality holds:
$$P(|F-E[F]|>x)\leq 2e^{-\frac{x^2}{2L^2}},\,\,\,x\geq 0.$$
This is one of the most fundamental inequalities in the theory of Gaussian processes and its proof can be found in many textbooks, see e.g. \cite{Adler2007}. Here our new concentration inequalities allow us to handle the class of functions with bounded partial derivatives of second order. We have the following exponential concentration inequality.
\begin{prop}Let $f:\mathbb{R}^n\to \mathbb{R}$ be a twice differentiable function satisfying
$$ \lambda_{ik}:=\sup\limits_{x\in \mathbb{R}^n}\big|\frac{\partial^2f(x)}{\partial x_k\partial x_i}\big|<\infty\,\,\forall\,1\leq i\leq k\leq n.$$
Then, it holds that
\begin{equation}\label{uao}
P\left(|F-E[F]|>x\right)\leq 4\exp\left(-\frac{2x^2}{\left(\sigma+\sqrt{\sigma^2+4\sqrt{c}\,x}\right)^2}\right),\,\,\,x\geq  0,
\end{equation}
where $\sigma^2={\rm Var}(F)$ and $c=\sum\limits_{k=1}^{n}\sum\limits_{i=1}^{k-1}\lambda_{ik}^2+\frac{1}{2}\sum\limits_{k=1}^{n}\lambda_{kk}^2.$
\end{prop}
\begin{proof}Without loss of generality, we may write $Z_k=B_k-B_{k-1},1\leq k\leq n.$ Then, by using the representation $F=f(B_1,B_2-B_1,\cdots,B_n-B_{n-1}),$ we deduce
\begin{align*}
\int_0^n\int_0^s |D_rE[D_sF|\mathcal{F}_s]|^2drds
&=\sum\limits_{k=1}^{n}\int_{k-1}^k\int_0^s |D_rE[\partial_kF|\mathcal{F}_s]|^2drds\\
&=\sum\limits_{k=1}^{n}\int_{k-1}^k\left(\sum\limits_{i=1}^{k-1}|E[\partial_{ik}F|\mathcal{F}_s]|^2+\int_{k-1}^s|E[\partial_{kk}F|\mathcal{F}_s]|^2 dr\right)ds\\
&\leq \sum\limits_{k=1}^{n}\int_{k-1}^k\left(\sum\limits_{i=1}^{k-1}\lambda_{ik}^2+\int_{k-1}^s\lambda_{kk}^2 dr\right)ds\\
&=\sum\limits_{k=1}^{n}\sum\limits_{i=1}^{k-1}\lambda_{ik}^2+\frac{1}{2}\sum\limits_{k=1}^{n}\lambda_{kk}^2.
\end{align*}
Here, to simplify notation, we used in the following abbreviations $\partial_kF=\frac{\partial f}{\partial x_k}(B_1,B_2-B_1,\cdots,B_n-B_{n-1})$ and $\partial_{ik}F=\frac{\partial^2f}{\partial x_k\partial x_i}(B_1,B_2-B_1,\cdots,B_n-B_{n-1}).$ On the other hand, by the Clark-Ocone formula, we have
$$F-E[F]=\int_0^n u_sdB_s,$$
where $u_s:=E[D_sF|\mathcal{F}_s],0\leq s\leq n.$ Furthermore, the stochastic process $\bar{u}_s=u_s$ satisfies $\int_0^nE[\bar{u}_s^2]ds={\rm Var}(F)=:\sigma^2$ and
$$\int_0^n\int_0^s |D_r\bar{u}_s|^2drds=\int_0^n\int_0^s |D_rE[D_sF|\mathcal{F}_s]|^2drds\leq \sum\limits_{k=1}^{n}\sum\limits_{i=1}^{k-1}\lambda_{ik}^2+\frac{1}{2}\sum\limits_{k=1}^{n}\lambda_{kk}^2=:c.$$
We now apply Corollary \ref{fdl} and we obtain
\begin{align*}
P(F-E[F]>x)&\leq P\left(\sup\limits_{0\leq t\leq n}\int_0^tu_sdB_s>x\right)\\
&\leq 2\exp\left(-\frac{2x^2}{\left(\sigma+\sqrt{\sigma^2+4\sqrt{c}\,x}\right)^2}\right),\,\,\,x\geq  0.
\end{align*}
Similarly, we also have
$$P(-F+E[F]>x)\leq 2\exp\left(-\frac{2x^2}{\left(\sigma+\sqrt{\sigma^2+4\sqrt{c}\,x}\right)^2}\right),\,\,\,x\geq  0.$$
So the desired inequality (\ref{uao}) follows.
\end{proof}

\subsection{Supremum of double Wiener-It\^o integrals}
Let $g$ be an integrable square symmetric function on $[0,T]^2.$ We consider the double Wiener-It\^o integrals
$$I_t(g):=\int_{[0,t]^2}g(s,\theta)dB_\theta dB_s,\,\,0\leq t\leq T.$$
The tail behaviour of multiple Wiener-It\^o integrals has been well studied, see e.g. \cite{Lata2006,Major2005}. Particularly, for the double Wiener-It\^o integrals, Theorem 4.1 in \cite{Major2005} provides us the concentration bound
$$P\left(|I_T(g)|>x\right)\leq C\exp\left(-\frac{x}{2\|g\|_{L^2([0,T]^2)}}\right),\,\,\,x\geq  0,$$
where $C$ is an absolute constant. For the supremum of double Wiener-It\^o integrals, the following tail estimate  seems to be new.
\begin{prop} It holds that
$$P\big(\sup\limits_{0\leq t\leq T}I_t(g)>x\big)\leq 2\exp\left(-\frac{x^2}{\left(\|g\|_{L^2([0,T]^2)}+\sqrt{\|g\|_{L^2([0,T]^2)}^2+\sqrt{8}\|g\|_{L^2([0,T]^2)}\,x}\right)^2}\right),\,\,\,x\geq  0.$$
\end{prop}
\begin{proof}
We consider the processes
$$\bar{u}_s=u_s=\int_0^sg(s,\theta)dB_\theta,\,\,0\leq s\leq T.$$
In view of Corollary \ref{fdl}, we have $\bar{\sigma}_T^2:=\int_0^TE|\bar{u}_s|^2ds=\int_0^T\int_0^sg^2(s,\theta)d\theta ds=\frac{1}{2}\|g\|_{L^2([0,T]^2)}^2$ and
$$\int_0^T\int_0^s |D_r\bar{u}_s|^2drds=\int_0^T\int_0^sg^2(s,r)dr ds=\frac{1}{2}\|g\|_{L^2([0,T]^2)}^2=:c\,\,a.s.$$
So, by using the bound (\ref{ydlj}), we deduce
\begin{align*}
P\big(\sup\limits_{0\leq t\leq T}I_t(g)>x\big)&=P\left(\sup\limits_{0\leq t\leq T}\int_0^t\int_0^sg(s,\theta)dB_\theta dB_s>x/2\right)\\
&\leq 2\exp\left(-\frac{x^2}{\left(\|g\|_{L^2([0,T]^2)}+\sqrt{\|g\|_{L^2([0,T]^2)}^2+\sqrt{8}\|g\|_{L^2([0,T]^2)}\,x}\right)^2}\right),\,\,\,x\geq  0.
\end{align*}
This completes the proof of the proposition.
\end{proof}
\begin{rem}
 It was brought to our attention (by Professor Christian Houdr\'e, personal communications) that we can use the results from \cite{Breton2007} to derive concentration inequalities for the supremum of double Wiener-It\^o integrals.
\end{rem}
\subsection{Supremum of  Cox-Ingersoll-Ross process}

The Cox-Ingersoll-Ross process is very popular in financial mathematics. It is the solution to the following equation
\begin{equation}\label{44kfo}
X_t=x_0+\int_0^t (a-bX_s)ds+\int_0^t \sigma\sqrt{X_s}dB_s,\,\,t\in [0,T],
\end{equation}
where the initial condition $x_0>0,$ the parameters $a,b,\sigma$ are positive constants. We assume that $2a>\sigma^2,$ which, by the Feller test, ensures the existence of a unique positive solution to the equation (\ref{44kfo}) (see, e.g. Chapter 5 in \cite{Karatzas1991}).

The asymptotic behavior of $P(\sup\limits_{0\leq t\leq T}X_t>x)$ as $x\to\infty$ has been recently discussed in \cite{Gerhold2022}. Here we obtain the following non-asymptotic estimate.
\begin{prop}It holds that
\begin{equation}
P\left(\sup\limits_{0\leq t\leq T}e^{bt}(X_t-E[X_t])>x\right)\leq 2\exp\left(-\frac{2x^2}{\left(\bar{\sigma}_T+\sqrt{\bar{\sigma}_T^2+\frac{ \sqrt{2}\sigma^2}{b}(e^{bT}-1)\,x}\right)^2}\right),\,\,\,x\geq  0,
\end{equation}
where $\bar{\sigma}_T^2:=\sigma^2\left(\frac{x_0}{b}(e^{bT}-1)+\frac{a}{2b^2}(e^{bT}-1)^2\right).$
\end{prop}
\begin{proof}
It is known from Lemma 4.1 in \cite{Alos2008} that the random variable $v_t:=\sqrt{X_t}$ is Malliavin differentiable and its derivative is given by
$$D_rv_t=\frac{ \sigma}{2}\exp\left[\int_r^t \left(-\left(\frac{a}{2}-\frac{\sigma^2}{8}\right)\frac{1}{v^2_s}-\frac{b}{2}\right)ds\right]\leq \frac{ \sigma}{2}e^{-\frac{b}{2}(t-r)},\,\,0\leq r\leq t\leq T.$$
We observe that $E[X_t]=x_0e^{-bt}+\frac{a}{b}(1-e^{-bt}),$ and
$$e^{bt}(X_t-E[X_t])=\int_0^t\sigma e^{bs}\sqrt{X_s}dB_s=\int_0^t\sigma e^{bs}v_sdB_s,\,\,t\in [0,T],$$
and hence,
\begin{equation}\label{8ts}
P\left(\sup\limits_{0\leq t\leq T}e^{bt}(X_t-E[X_t])>x\right)=P\left(\sup\limits_{0\leq t\leq T}\int_0^t\sigma e^{bs}v_sdB_s>x\right),\,\,x\geq 0.
\end{equation}
We consider the stochastic process $\bar{u}_s:=\sigma e^{bs}v_s,0\leq s\leq T.$ We have
$$\int_0^T\int_0^s |D_r\bar{u}_s|^2drds\leq \int_0^T\int_0^s\frac{ \sigma^4}{4} e^{2bs}e^{-b(s-r)}drds
=\frac{ \sigma^4}{8b^2}(e^{bT}-1)^2\,\,a.s.$$
On the other hand, we also have
$$\bar{\sigma}_T^2:=\int_0^TE[\bar{u}_s^2]ds=\int_0^T\sigma^2 e^{2bs}E[X_s]ds=\sigma^2\left(\frac{x_0}{b}(e^{bT}-1)+\frac{a}{2b^2}(e^{bT}-1)^2\right).$$
We now apply Corollary \ref{fdl} to $u_s=\sigma e^{bs}v_s,0\leq s\leq T,$ and we obtain
$$P\left(\sup\limits_{0\leq t\leq T}\int_0^tu_sdB_s>x\right)\leq 2\exp\left(-\frac{2x^2}{\left(\bar{\sigma}_T+\sqrt{\bar{\sigma}_T^2+\frac{ \sqrt{2}\sigma^2}{b}(e^{bT}-1)\,x}\right)^2}\right),\,\,\,x\geq  0.$$
So, recalling (\ref{8ts}), the proof of the proposition is complete.
\end{proof}


\noindent {\bf Acknowledgments.}  The authors would like to thank the anonymous referees for valuable comments which led to the improvement of this work.

\end{document}